\documentclass[a4paper]{amsart}
\usepackage{graphicx}
\usepackage{amssymb}
\usepackage{amsmath}
\usepackage{amsthm,amsfonts,bbm}
\usepackage{amscd}
\usepackage[all,2cell]{xy}

\UseAllTwocells \SilentMatrices

\newtheorem{thm}{Theorem}[section]

\newtheorem{lem}[thm]{Lemma}

\newtheorem{prop-def}[thm]{Proposition-Definition}
\theoremstyle{definition}
\newtheorem{defi}[thm]{Definition}
\theoremstyle{remark}

\newtheorem{exm}[thm]{\bf Example}
\numberwithin{equation}{section}
\numberwithin{figure}{section}

\newcommand{\A}{\mathcal{A}}

\def\la{\lambda}

\def\A{\mathcal{A}}

\def\PS{\mathbb{S}}

\def\diag{{\rm diag}}
\def\x{{\mathbf x}}
\def\y{{\mathbf y}}

\def\v{{\mathbf v}}
\def\C{\mathbb{C}}

\def\Z{\mathbb{Z}}

\def \PV{\mathbb{V}}

\def\I{\mathcal{I}}
\def\B{\mathcal{B}}

\def\diag{{\rm diag}}

\def \A{\mathcal{A}}
\def \L{\mathcal{L}}
\def \Q{\mathcal{Q}}
\def \Dg{\mathfrak{D}}
\def \D{\mathcal{D}}

\def \sp{\mbox{\rm supp}}

\def \inc{\mbox{\rm I}}
\def \lamin{\lambda_{\min}}

\begin{document}
\title[Eigenvectors of $Z$-tensors]
{Eigenvectors of $Z$-tensors associated with least H-eigenvalue with application to hypergraphs}

\author[Y.-Z. Fan]{Yi-Zheng Fan$^*$}
\address{School of Mathematical Sciences, Anhui University, Hefei 230601, P. R. China}
\email{fanyz@ahu.edu.cn}
\thanks{$^*$The corresponding author.
This work was supported by National Natural Science Foundation of China (Grant No. 11871073, 11771016, 11871071).}

\author[Y. Wang]{Yi Wang}
\address{School of Mathematical Sciences, Anhui University, Hefei 230601, P. R. China}
\email{wangy@ahu.edu.cn}

\author[Y.-H. Bao]{Yan-Hong Bao}
\address{School of Mathematical Sciences, Anhui University, Hefei 230601, P. R. China}
\email{baoyh@ahu.edu.cn}

%
%






\begin{abstract}
Unlike an irreducible $Z$-matrices, a weakly irreducible $Z$-tensor $\A$ can have more than one eigenvector associated with the least H-eigenvalue.
We show that there are finitely many eigenvectors of $\A$ associated with the least H-eigenvalue.
If $\A$ is further combinatorial symmetric, the number of such eigenvectors can be obtained explicitly by the Smith normal form of the incidence matrix of $\A$.
When applying to a connected uniform hypergraph $G$, we prove that the number of Laplacian eigenvectors of $G$ associated with the zero eigenvalue
is equal to the the number of adjacency eigenvectors of $G$ associated with the spectral radius, which is also equal to
the number of signless Laplacian eigenvectors of $G$ associated with the zero eigenvalue if zero is an signless Laplacian eigenvalue.
\end{abstract}

\maketitle

\section{Introduction}
A real \emph{tensor} (also called \emph{hypermatrix}) $\A=(a_{i_{1} i_2 \ldots i_{m}})$ of order $m$ and dimension $n$ refers to a
 multiarray of entries $a_{i_{1}i_2\ldots i_{m}}\in \mathbb{R}$ for all $i_{j}\in [n]:=\{1,2,\ldots,n\}$ and $j\in [m]$, which can be viewed to be the coordinates of the classical tensor under an orthonormal basis.
For a vector $\x=(x_1, \cdots, x_n)\in \C^n$, $\A\x^{m-1}\in \C^n$, which is defined as
\begin{align*}
(\A\x^{m-1})_i=\sum_{i_2,\cdots, i_m\in [n]}a_{ii_2\cdots i_m}x_{i_2}\cdots x_{i_m}, i\in [n].
\end{align*}
Let $\I=(i_{i_1i_2\cdots i_m})$ be the \emph{identity tensor} of order $m$ and dimension $n$, that is,
$i_{i_1i_2\cdots i_m}=1$ if $i_1=i_2=\cdots=i_m$ and $i_{i_1i_2\cdots i_m}=0$ otherwise.

\begin{defi}[\cite{Lim,Qi,CPZ2}] Let $\A$ be an $m$-th order $n$-dimensional tensor.
For some $\lambda \in \mathbb{C}$, if the polynomial system $(\lambda \mathcal{I}-\A)x^{m-1}=0$, or equivalently $\A x^{m-1}=\lambda x^{[m-1]}$, has a solution $x\in \mathbb{C}^{n}\backslash \{0\}$,
then $\lambda $ is called an \emph{eigenvalue} of $\A$ and $x$ is an \emph{eigenvector} of $\A$ associated with $\lambda$,
where $x^{[m-1]}:=(x_1^{m-1}, x_2^{m-1},\ldots,x_n^{m-1})$.
\end{defi}

For a real tensor $\A$, an eigenvalue $\la$ of $\A$ is called an \emph{H-eigenvalue} if
there is a real eigenvector associated with $\la$, implying that $\la$ is real.
Denote by $\lamin(\A)$ the least H-eigenvalue of $\A$, and $\rho(\A)$ the \emph{spectral radius} of $\A$ (i.e. the largest modulus of the eigenvalues of $\A$).
Let $\mathbb{P}^{n-1}$ be the complex projective spaces of dimension $n-1$.
Consider the projective variety
$$\PV_\la=\PV_\la(\A)=\{\x \in \mathbb{P}^{n-1}: \A\x^{m-1}=\la \x^{[m-1]}\}.$$
which is called the \emph{projective eigenvariety} of $\A$ associated with $\la$ \cite{FBH}.
In this paper the number of eigenvectors of $\A$ is considered in $\PV_\la(\A)$.

By the Perron-Frobenius theorem of nonnegative tensors \cite{CPZ,FGH,YY1,YY2,YY3},
for an irreducible (or weakly irreducible) nonnegative tensor $\A$, the spectral radius $\rho(\A)$ is an eigenvalue of $\A$ associated with a unique nonnegative eigenvector (or positive eigenvector) up to a scalar, which is called the \emph{Perron vector} of $\A$.
In \cite{FHB} the authors investigate the spectral symmetry of $\A$ by using the eigenvalues with modulus $\rho(\A)$,
which generalizes some spectral properties of nonnegative irreducible matrices.
But, different from the matrices case, $\A$ can have more than one eigenvector associated with $\rho(\A)$, including the Perron vector.

$Z$-matrices and $M$-matrices are the generalization of nonnegative matrices.
Recently they were generalized to $Z$-tensors and $M$-tensors respectively \cite{ZQZ,DQW}.

\begin{defi}[\cite{ZQZ,DQW}]\label{ZM}
A real tensor $\A$ is called a \emph{$Z$-tensor} if all of its off-diagonal entries are non-positive, or equivalently it can be written as
\begin{equation}\label{ZMe}
\A =s \I -\B,
\end{equation}
where $s>0$ and $\B$ is nonnegative.
If $s \ge \rho(\B)$, then $\A$ is called an \emph{$M$-tensor}; if $s > \rho(\B)$, then $\A$ is called a \emph{nonsingular $M$-tensor} or \emph{strong $M$-tensor}.
\end{defi}

Zhang et al. \cite{ZQZ} showed that the minimum real part of all eigenvalues an $M$-tensor is the least H-eigenvalue.
Some characterization of a $Z$-tensor being an $M$-tensor is given in \cite{ZQZ, DQW}.
Most of these results are generalization of $M$-matrices.
Let $\A$ be a weakly irreducible $Z$-tensor as in (\ref{ZMe}).
Then $ \PV_{\lamin(\A)}=\PV_{\rho(\B)}$.
As $\B$ is nonnegative and weakly irreducible, it can have more than one eigenvector associated with $\rho(\B)$, which implies that
$\A$ can have more than one eigenvector associated with $\lamin(\A)$.
This is different from the case of irreducible $Z$-matrices.

In this paper, we show that for a weakly irreducible $Z$-tensor $\A$, $\PV_{\lamin(\A)}$ is a finite abelian group, and admits a $\Z_m$-module if $\A$ is further combinatorial symmetric.
For the latter case, the cardinality of $\PV_{\lamin(\A)}$ can obtained explicitly by solving the Smith normal form of the incidence matrix of $\A$.
For a connected uniform hypergraph $G$, we show that $\PV_0(\L(G))$ has the same cardinality as $\PV_{\rho(\A(G))}$,
which is also equal to that of $\PV_0(\Q(G))$ if $0$ is an eigenvalue of $\Q(G)$, where $\A(G)$, $\L(G)$ and $\Q(G)$ are the adjacency, Laplaican and
signless Laplacian of $G$ respectively.

\section{Preliminaries}
Let $\A$ be an $m$-th order $n$-dimensional real tensor.
$\A$ is called \textit{symmetric} if its entries are invariant under any permutation of their indices.
The irreducibility or weakly irreducibility of a tensor can be referred to \cite{CPZ,FGH}.
The {\it support} of $\A$, denoted by $\sp(\A)=(s_{i_1i_2\ldots i_m})$, is defined as a tensor with same order and dimension as $\A$, such that $s_{i_1\ldots i_m}=1$ if $a_{i_1\ldots i_m} \ne 0$,
and $s_{i_1\ldots i_m}=0$ otherwise.
$\A$ is called {\it combinatorial symmetric} if $\sp(\A)$ is symmetric.
Let $\A$ be a combinatorial symmetric tensor of order $m$ and dimension $n$.
Set $$E(\A)=\{(i_1, i_2, \cdots, i_m)\in [n]^m: a_{i_1i_2\cdots i_m}\neq 0, 1\le i_1\le \cdots \le i_m \le n\}.$$
Define
\[b_{e,j}=|\{k: i_k=j, e=(i_1, i_2, \cdots, i_m) \in E(\A), k \in [m]\}|\]
and obtain an $|E(\A)|\times n$ matrix $\inc_\A=(b_{e,j})$, called the \emph{incidence matrix} of $\A$.

A \emph{hypergraph} $G=(V(G),E(G))$ consists of a vertex set  $V(G)=\{v_1, v_2, \ldots, v_n\}$
and an edge set $E(G)=\{e_1, e_2, \ldots, e_{l}\}$, where $e_{j}\subseteq V(G)$ for $j \in [l]$.
If $|e_{j}|=m$ for each $j\in [l]$, then $G$ is called an {\it $m$-uniform} hypergraph.
The {\it adjacency tensor} of an $m$-uniform hypergraph $G$ is defined to be $\mathcal{A}(G)=(a_{i_{1}i_{2}\ldots i_{m}})$,
where $a_{i_{1}i_{2}\ldots i_{m}}=\frac{1}{(m-1)!}$ if $\{v_{i_{1}},v_{i_{2}},\ldots,v_{i_{m}}\} \in E(G)$, and is $0$ otherwise \cite{CD}.
Let $\mathcal{D}(G)$ be an $m$-th order $n$-dimensional diagonal tensor,
    where $d_{i\ldots i}=d_{v_i}$, the degree of the vertex $v_i$, for each $i \in [n]$.
Then $\L(G)=\mathcal{D}(G)-\A(G)$ and $\Q(G)=\mathcal{D}(G)+\A(G)$ are called the {\it Laplacian tensor} and the {\it signless Laplacian tensor} of $G$, respectively \cite{Qi3}.
The adjacency, Laplacian or signless Laplacian tensor of $G$ is symmetric, and it is weakly irreducible if and only if $G$ is connected \cite{PZ, YY3}.
The \emph{incidence matrix} of $G$, denoted by $\inc_G=(b_{e,v})$, coincides with that of $\A(G)$, namely
$b_{e,v}=1$ if $v\in e$, and $b_{e,v}=0$ otherwise.

For a matrix $B \in \Z_m^{k \times n}$,
there exist two invertible matrices $P \in \Z_m^{k \times k}$ and $Q \in \Z_m^{n \times n}$ such that
\begin{equation} \label{smith}
PBQ=\begin{pmatrix}
d_1 & 0 & 0 &  & \cdots & & 0\\
0 & d_2 & 0 &  & \cdots & &0\\
0 & 0 & \ddots &  &  & & 0\\
\vdots &  &  & d_r &  & & \vdots\\
 & & & & 0 & & \\
  & & & &  & \ddots & \\
0 &  &  & \cdots &  & &0
\end{pmatrix},
\end{equation}
where $r \ge 0$, $ 1 \le d_i \le m-1$, $d_i | d_{i+1}$ for $i \in [r-1]$, and $d_i |m$ for all $i \in [r]$.
The matrix in (\ref{smith}) is called the {\it Smith normal form} of $B$ over $\Z_m$.

Let $\A$ be a tensor of order $m$ and dimension $n$.
Define
\begin{equation} \label{Dg}
\Dg^{(0)}=\Dg^{(0)}(\A)=\{D: \A=D^{-(m-1)}\A D, d_{11}=1\},
\end{equation}
where $D$ is an $n \times n$ invertible diagonal matrix and the product is defined as in \cite{Shao}.

\begin{defi}[\cite{FBH}]
For a general tensor $\A$, the cardinality of $\Dg^{(0)}(\A)$, denoted by $s(\A)$, is called the \emph{stabilizing index} of $\A$.
\end{defi}

By \cite[Lemma 2.5(1)]{FBH}, $\Dg^{(0)}(\A)$ is an abelian group under the usual matrix multiplication,
which is determined by the support of $\A$ by \cite[Lemma 2.6]{FBH}.
Suppose that $\A$ is a nonnegative weakly irreducible tensor.
By \cite[Theorem 3.7]{YY3}, for each $\y \in \PV_{\rho(\A)}$, $|\y|=:\v_p$ is the unique positive Perron vector of $\A$.
Therefore, we can assume each $\y \in \PV_{\rho(\A)}$ satisfies $y_1=1$.
Define
\begin{equation}\label{Dy}
D_\y=\diag\left(y_1/|y_1|,\ldots,y_n/|y_n|\right),
\end{equation}
and a quasi-Hadamard product $\circ$ in $\PV_{\rho(\A)}$ as follows:
\begin{equation}\label{product} \y \circ \hat{\y}:=D_\y D_{\hat{\y}} \v_p.\end{equation}

\begin{lem}[\cite{FBH}, Lemma 2.5, Lemma 3.1] \label{VtoD}
Let $\A$ be a nonnegative weakly irreducible tensor.
Then the following results hold.
\begin{enumerate}
\item $\Dg^{(0)}(\A)=\{D_{\y}: \y \in \PV_{\rho(\A)}, y_1=1\}$, and hence $s(\A)=|\PV_{\rho(\A)}|$.

\item $(\PV_{\rho(\A)},\circ)$ is an abelian group isomorphism to $\Dg^{(0)}(\A)$.
\end{enumerate}
\end{lem}

Further assume $\A$ is also combinatorial symmetric of order $m$.
By \cite[Lemma 2.5(3)]{FBH}, $D^m=\I$ for each $D \in \Dg^{(0)}$.
Hence for each $\y \in \PV_{\rho(\A)}$, $\y^{\circ m}=D_{\y}^m \v_p=\v_p$ (the identity), which implies that
$(\PV_{\rho(\A)}, \circ)$ admits a $\Z_m$-module.
Define
\begin{equation}\label{ps0} \PS_0(\A)=\{\x \in \Z_m^n: \inc_\A\x=0 \hbox{~over~} \Z_m, x_1=0\},\end{equation}
where $\inc_\A$ is the incidence matrix of $\A$.

\begin{thm}[\cite{FBH}]\label{VtoS}
Let $\A$ be a nonnegative combinatorial symmetric weakly irreducible tensor of order $m$.
Then $\PV_{\rho(\A)}$ is $\Z_m$-module isomorphic to $\PS_0(\A)$.
\end{thm}

\begin{thm}[\cite{FBH}] \label{stru}
Let $\A$ be a nonnegative combinatorial symmetric weakly irreducible tensor of order $m$ and dimension $n$.
Suppose that $\inc_\A$ has a Smith normal form over $\Z_m$ as in (\ref{smith}).
Then $1 \le r \le n-1$, and
\begin{equation}\label{struEQ}
\PS_0(\A) \cong \oplus_{i, d_i \ne 1} \Z_{d_i} \oplus (n-1-r)\Z_m.
\end{equation}
\end{thm}


\begin{thm}[\cite{FHB2}]\label{fint}
Let $\A$ be a weakly irreducible nonnegative tensor.
Then $\PV_{\rho(\A)}$ has dimension zero, i.e. there are finite many eigenvectors of $\A$ corresponding to $\rho(\A)$ up to a scalar.
\end{thm}

\section{$Z$-tensors, Laplacian tensors and signless Laplacian tensors}

\begin{lem}\label{leH}
Let $\A=s \I -\B$ be a $Z$-tensor, where $s >0$ and $\B \ge 0$.
Then
\begin{enumerate}
\item  $\lamin(\A)=s-\rho(\B)$, which is the eigenvalue of $\A$ with the least real part.

\item $\PV_{\lamin(\A)}=\PV_{\rho(\B)}$.

\item $\Dg^{(0)}(\A)=\Dg^{(0)}(\B)$, $s(\A)=s(\B)$.

\item For any diagonal tensor $\D$, $s(\A)=s(\D+\A)$.
\end{enumerate}
\end{lem}

\begin{proof}
The first result follows by a similar discussion to \cite[Theorem 3.3]{ZQZ}.
Obviously, $\x$ is an eigenvector of $\A$ associated with $\lamin(\A)$ if and only if it is an eigenvector of $\B$ associated with $\rho(\B)$.
So the second result follow.
The last two results follow from the definition.
\end{proof}

\begin{lem}\label{plusD}
Let $\A$ be a weakly irreducible $Z$-tensor. Then
\begin{enumerate}
\item $\PV_{\lamin(\A)}$ is finite, i.e. there are finitely many eigenvectors of $\A$ associated with $\lamin(\A)$.

\item $\PV_{\lamin(\A)}$ is an abelian group isomorphic to $\Dg^{(0)}(\A)$.

\item $s(\A)=|\PV_{\lamin(\A)}|$.

\end{enumerate}
\end{lem}

\begin{proof}
Suppose $\A=s \I -\B$, where $s>0$ and $\B \ge 0$.
Then $\PV_{\lamin(\A)}=\PV_{\rho(\B)}$ by Lemma \ref{leH}.
As $\B$ is nonnegative and weakly irreducible, the finiteness of $\PV_{\lamin(\A)}$ follows from Theorem \ref{fint}.
By Lemma \ref{VtoD}, $\PV_{\rho(\B)}$, as well as $\PV_{\lamin(\A)}$, is an abelian group isomorphic to $\Dg^{(0)}(\B)$, which is equal to $\Dg^{(0)}(\A)$ by Lemma \ref{leH}.
\end{proof}

\begin{thm}\label{struZ}
Let $\A$ be a combinatorial symmetric weakly irreducible $Z$-tensor of order $m$ and dimension $n$.
Suppose that $\inc_\A$ has a Smith normal form over $\Z_m$ as in (\ref{smith}).
Then
\begin{enumerate}
\item $1 \le r \le n-1$.

\item $\PV_{\lamin(\A)}$ is a $\Z_m$-module with the decomposition
\[\PV_{\lamin(\A)} \cong \oplus_{i, d_i \ne 1} \Z_{d_i} \oplus (n-1-r)\Z_m.\]

\item  $s(\A)=|\PV_{\lamin(\A)}|=m^{n-1-r} \Pi_{i=1}^r d_i$.
\end{enumerate}

\end{thm}

\begin{proof}
Let $\A=s \I -\B$, where $s>0$ and $\B \ge 0$.
Note that $\B$ is combinatorial symmetric and weakly irreducible,
$\A$ and $\B$ have the same incidence matrices, i.e. $\inc_\A=\inc_\B$ over $\Z_m$.
So, by Theorem \ref{stru}, $1 \le r \le n-1$.
By Theorem \ref{VtoS},  $\PV_{\rho(\B)}$ is $\Z_m$-module isomorphic to $\PS_0(\B)$, which has a decomposition as in (\ref{struEQ}).
As $\PV_{\lamin(\A)}=\PV_{\rho(\B)}$ by Lemma \ref{leH}, the second result follows.
The last result follows from (2) and Lemma \ref{plusD}.
\end{proof}


\begin{lem}
The Laplacian tensor $\L(G)$ of a uniform hypergraph $G$ is a singular $M$-tensor.
\end{lem}

\begin{proof}
As $\L(G)$ is diagonal dominant, $\L(G)$ is an $M$-tensor by \cite[Theorem 3.15]{ZQZ}.
Note that $0$ is the least $H$-eigenvalue of $\L(G)$ associated with an all-ones eigenvector.
So $\L(G)$ is singular.
\end{proof}

\begin{thm}\label{numL}
Let $G$ be a connected $m$-uniform hypergraph.
Suppose that $\inc_G$ has a Smith normal form over $\Z_m$ as in (\ref{smith}).
Then $$|\PV_0(\L(G))|=s(\L(G))=s(\A(G))=m^{n-1-r} \Pi_{i=1}^r d_i.$$
\end{thm}

\begin{proof}
As $G$ is connected, $\L(G)$ is a combinatorial symmetric weakly irreducible $Z$-tensor.
The incidence matrix of $\L(G)$ is same as that of $\A(G)$, i.e. $\inc_G$ over $\Z_m$.
So, by Theorem \ref{struZ}(3), $s(\L(G))=|\PV_0(\L(G))|=m^{n-1-r} \Pi_{i=1}^r d_i$.
The result now follows as $s(\L(G))=s(\A(G))$ by definition.
\end{proof}

Let $G$ be an $m$-uniform hypergraph on $n$ vertices, where $m$ is even.
$G$ is called \emph{odd-colorable} \cite{Ni} if there exists a map
$f: V(G) \to [m]$ such that for each $e \in E(G)$,
$\sum_{v \in e} f(v)\equiv \frac{m}{2} \mod m$.

\begin{lem}[\cite{FWBWLZ}]\label{char-odd}
Let $G$ be an $m$-uniform connected hypergraph on $n$ vertices.
Then the following are equivalent.

\begin{enumerate}

\item  $0$ is an eigenvalue of $\Q(G)$.

\item  $m$ is even and $G$ is odd-colorable.

\item  $\Q(G)=D^{-(m-1)}\L(G)D$ for some diagonal matrix $D$ with $|D|=\I$.

\item  $\rho(\L(G))=\rho(\Q(G))$.
\end{enumerate}
\end{lem}

\begin{thm}\label{numQ}
Let $G$ be an $m$-uniform connected hypergraph which is odd-colorable.
Suppose that $\inc_G$ has a Smith normal form over $\Z_m$ as in (\ref{smith}).
Then
\begin{enumerate}
\item $s(\Q(G))=s(\L(G))=s(\A(G))=m^{n-1-r} \Pi_{i=1}^r d_i$.

\item $s(\Q(G))=|\PV_{\rho(\Q(G))}|=|\PV_0(\Q(G))|=|\PV_0(\L(G))|$.
\end{enumerate}
\end{thm}

\begin{proof}
The first result follows by definition and Theorem \ref{stru}.
By Lemma \ref{char-odd}, as $G$ is odd-colorable, $0$ is an eigenvalue of $\Q(G)$.
Also by Lemma \ref{char-odd}, there exists an invertible  diagonal matrix $D$ such that $\y$ is an eigenvector of $\Q(G)$ associated with $0$ if and only if
$D\y$ is an eigenvector of $\L(G)$ associated with $0$.
So $$|\PV_0(\Q(G))|=|\PV_0(\L(G))|=s(\L(G)),$$
where the last equality follows from Lemma \ref{plusD}.
By Lemma \ref{VtoD}, $s(\Q(G))=|\PV_{\rho(\Q(G))}|$ as $\Q(G)$ is nonnegative and weakly irreducible.
The result follows.
\end{proof}

\begin{exm}
Let $K_n^{[m]}$ be a complete $m$-uniform hypergraph on $n$ vertices, where $n \ge m+1$.
By \cite[Example 4.4]{FBH}, $s(\A(K_n^{[m]})=1$, so $s(\L(K_n^{[m]}))=1$ by Theorem \ref{numL}, which implies that
$\L(K_n^{[m]})$ has only one eigenvector (the all-ones vector) associated with the zero eigenvalue.
\end{exm}

An \emph{odd bipartition} $\{V_1,V_2\}$ of $G$ is a bipartition of $V(G)$
such that each edge of $G$ intersects with $V_1$ or $V_2$ in an odd number of vertices.
$G$ is called \emph{odd-bipartite} if $G$ has an odd bipartition \cite{HQ}.
Shao et al. \cite{SSW} proved that $0$ is an $H$-eigenvalue of $\Q(G)$ if and only if
$m$ is even and $G$ is odd-bipartite.
An odd-bipartite hypergraph is odd-colorable, but the converse is not true \cite{Ni}.
A \emph{cored hypergraph} \cite{HQS} is one such that each edge contains a vertex of degree one.
Obviously a cored hypergraph of even uniformity is odd-bipartite.

\begin{exm}
Let $G$ be a connected $m$-uniform cored hypergraph on $n$ vertices with $t$ edges.
Then $s(\A(G))=m^{n-1-t}$ by \cite[Theorem 4.1]{FBH}. So $\L(G)$ has $m^{n-1-t}$ eigenvectors associated with the zero eigenvalue.
If $m$ is even, then $G$ is odd-bipartite.
So $\Q(G)$ has $m^{n-1-t}$ eigenvectors associated with the zero eigenvalue.
\end{exm}

\begin{exm}
Let $G^{m,m/2}$ be a generalized power hypergraph \cite{KF}, which is obtained from a simple graph $G$
 by blowing each vertex into an $m/2$-set and preserving the adjacency relation, where $m$ is even.
It is known that $G^{m,m/2}$ is non-odd-bipartite if and only if $G$ is non-bipartite \cite{KF}.
Suppose that $G$ is non-bipartite.
Then $\rho(\L(G^{m,m/2}))=\rho(\Q(G^{m,m/2}))$ if and only if $4$ divides $m$ \cite{FKT}.
Particularly, let $G$ be a triangle or $C_3$ and $m=4$.
Then $C_3^{4,2}$ is non-odd-bipartite but odd-colorable by Lemma \ref{char-odd}.
The incidence matrix of $C_3^{4,2}$ has invariant divisors $1,1,2$ over $\Z_4$.
So $\Q(C_3^{4,2})$ has $32$ eigenvectors associated with the zero eigenvalue.
\end{exm}

\end{document}